\documentclass[12pt]{amsart}
\usepackage{geometry,amsmath,amssymb,enumerate,fullpage}

\newcommand{\mathi}{\mathrm{i}}

\newtheorem{Conjecture}{Conjecture}
\newtheorem{corollary}{Corollary}
\newtheorem{lemma}{Lemma}
\newtheorem{theorem}{Theorem}
\newtheorem{proposition}{Proposition}
\newtheorem*{Main}{Main Theorem}

\begin{document}
\title{Gaps between zeros of $\zeta(s)$ 
and the distribution of zeros of $\zeta'(s)$}
\author{Maksym Radziwi\l\l}
\address{Department of Mathematics \\ Stanford University\\
450 Serra Mall, Bldg. 380\\
Stanford, CA 94305-2125}
\email{maksym@stanford.edu}
\thanks{The author is partially supported by a NSERC PGS-D award}
\begin{abstract}
We settle a conjecture of Farmer and Ki in a stronger form.
Roughly speaking we show that there is a positive proportion of small
gaps between consecutive zeros of the zeta-function $\zeta(s)$
if and only if there is a positive proportion of zeros of $\zeta'(s)$
lying very closely to the half-line.
Our work has applications to the Siegel zero problem. We provide
a criterion for the non-existence of the Siegel zero, solely in
terms of the distribution of the zeros of $\zeta'(s)$. Finally
on the Riemann Hypothesis and the Pair Correlation Conjecture we
obtain near optimal bounds for the number of zeros of $\zeta'(s)$
 lying very closely to the half-line.
Such bounds are relevant to a deeper understanding of Levinson's method, 
allowing us to place
one-third of the zeros of the Riemann zeta-function on the half-line.
\end{abstract}
\maketitle
\section{Introduction.}

The inter-relation
between the \textit{horizontal} distribution of zeros 
of $\zeta(s)$ (denoted $\rho = \beta + i\gamma$) 
and the \textit{horizontal} distribution of the zeros of
$\zeta'(s)$ (denoted $\rho' = \beta' + i\gamma'$) 
is the basis of Levinson's method
\cite{Levinson} allowing us to place 
one third of the zeros of $\zeta(s)$ on the
critical line. 

Recently it has been understood that the \textit{horizontal}
distribution of the zeros of $\zeta'(s)$ is also 
related to the \textit{vertical}
distribution of zeros of $\zeta(s)$. As an first attempt at capturing 
such a relationship
we have the following conjecture of Soundararajan \cite{Soundararajan}. 

\textit{Note:} Throughout we assume the Riemann Hypothesis. 
We recall that $\beta' \geq \tfrac 12$ for all non-trivial zeros of
$\zeta'(s)$ (see \cite{Speiser}) and that this is equivalent to the
Riemann Hypothesis.
\begin{Conjecture}[Soundararajan \cite{Soundararajan}]
We have
\begin{equation}
\tag{A}\label{eqn1}\stepcounter{equation}
\liminf_{\gamma \rightarrow \infty} (\gamma^+ - \gamma) \log \gamma = 0
\end{equation}
with $\gamma^+$ the ordinate suceeding $\gamma$, 
if and only if
\begin{equation}
\tag{B} \label{eqn2}\stepcounter{equation}
\liminf_{\gamma' \rightarrow \infty} (\beta' - \tfrac 12) \log \gamma' = 0
\end{equation}
\end{Conjecture}

Zhang \cite{Zhang} shows that $A \implies B$ (see also
\cite{Garaev} for a partial converse). 
Ki \cite{Ki} obtained
a necessary and sufficient condition for the negation of $B$. 
Ki's result shows that
zeros $\rho'$ with $(\beta' - \tfrac 12)\log \gamma' = o(1)$
arise not only from small gaps between zeros of $\zeta(s)$
but also, for example, from clusters of regularly spaced zeros of $\zeta(s)$. 
Therefore given our current knowledge about the zeros of $\zeta(s)$
it is possible for $B$ and the negation of $A$ to co-exist. 
The assertion $A$ is arithmetically very interesting, since, 
following an idea of Montgomery (made explicit by Conrey and Iwaniec
in \cite{ConreyIwaniec})
if there are many small gaps between consecutive zeros of
$\zeta(s)$ then the class number of $\mathbb{Q}(\sqrt{-d})$ is large
and there are no Siegel zeros.
%
%

A more recent attempt at capturing the relation between
the distribution of zeros of $\zeta(s)$ and $\zeta'(s)$ is
due to Farmer and Ki \cite{FarmerKi}.
Let $w(x)$ be the indicator function of the unit interval. 
Following Farmer and Ki we introduce two distribution functions,
\begin{align*}
m'(\varepsilon) & := \liminf_{T \rightarrow \infty} \frac{2\pi}{T \log T}
\sum_{T \leq \gamma' \leq 2T} w \bigg ( \frac{(\beta' - \tfrac 12)\log
T}{\varepsilon} \bigg ) \\
m(\varepsilon) & := \liminf_{T \rightarrow \infty} \frac{2\pi}{T \log T}
\sum_{T \leq \gamma \leq 2T} w \bigg ( \frac{(\gamma^+ - \gamma) \log T}{
\varepsilon} \bigg ).
\end{align*}
These are indeed distribution functions, since in a rectangle of
length $T$, both $\zeta(s)$ and $\zeta'(s)$ have asymptotically
$N(T) \sim (T/2\pi) \log T$ zeros (see \cite{Berndt}), and it is
conjectured 
that $m'(v) \rightarrow 1$ as $v \rightarrow \infty$,
whereas it is known that $m(v) \rightarrow 1$ as $v \rightarrow \infty$
(see \cite{Soundararajan}, \cite{Fujii}).

Zhang shows in \cite{Zhang} that if $m(\varepsilon) > 0$
for all $\varepsilon > 0$, then $m'(\varepsilon) > 0$.
An analogue of Soundararajan's conjecture would assert that
$m(\varepsilon) > 0$ for all $\varepsilon > 0$ if and only if
$m'(\varepsilon) > 0$ for all $\varepsilon > 0$. 
As explained by
Farmer and Ki in \cite{FarmerKi} if for example the zeros are well-spaced with
sporadic large gaps, something we cannot rule out at present,
then in principle $m'(\varepsilon) > 0$ is not enough to imply
$m(\varepsilon) > 0$. 
%
Farmer and Ki propose the following alternative
conjecture.
\begin{Conjecture}[Farmer and Ki \cite{FarmerKi}]  
If $m' (\varepsilon) \gg \varepsilon^v$ with a $v < 2$ as
$\varepsilon \rightarrow 0$ then
$m (\varepsilon) > 0$ for all $\varepsilon > 0$. 
\end{Conjecture}
This is a realistic conjecture since we expect that $m'(\varepsilon)
\sim (8/9\pi) \varepsilon^{3/2}$ as $\varepsilon \rightarrow 0$ (see \cite{Mezzadri}).
Farmer and Ki comment ``we intend this as a general
conjecture, applying to the Riemann zeta-function but also to
other cases such as a sequence of polynomials with all zeros on the
unit circle'' and that ``stronger statements should be true for the
zeta function''. 
Our main result is a proof of Conjecture 2 in 
a stronger and quantitative form for the Riemann zeta-function.
\begin{Main} \label{thm1}
Let $A,\delta > 0$ be given.
\begin{itemize}
\item If $m'(\varepsilon) \gg \varepsilon^{A}$
as $\varepsilon \rightarrow 0$ then $m(\varepsilon^{1/2}) \gg
\varepsilon^{A + \delta}$ for all $\varepsilon \leq 1$.  
\item If $m(\varepsilon^{1/2}) \gg
\varepsilon^{A}$ as $\varepsilon \rightarrow 0$ then $m'(\varepsilon)
\gg \varepsilon^{A + \delta}$ for all $\varepsilon \leq 1$.
\end{itemize} 
\end{Main}
We conjecture that $m'(\varepsilon) \asymp
m(\varepsilon^{1/2})$ provided that one of $m(\varepsilon)$ or $m'(\varepsilon)$
is $\gg \varepsilon^{A}$ for some $A > 0$. This is consistent with the 
expectation that $m(\varepsilon) \sim (\pi / 6) \varepsilon^3$ and 
$m'(\varepsilon) \sim (8/9\pi) \varepsilon^{3/2}$ as $\varepsilon \rightarrow 0$
(see \cite{Mezzadri}).
Our Main Theorem could be restated as saying that 
$$
\log m(\varepsilon) \sim \log m'(\varepsilon^{1/2})
$$
as $\varepsilon \rightarrow 0$ provided that one of $m(\varepsilon)$ or
$m'(\varepsilon)$ is greater than $\varepsilon^A$.
As a consequence of the Main Theorem we obtain 
estimates for $m'(\varepsilon)$
assuming the Pair Correlation
Conjecture. 
\begin{corollary}
Assume the Pair Correlation Conjecture. Let $\delta > 0$.
Then 
$$
\varepsilon^{3/2 + \delta} \ll m'(\varepsilon) \ll \varepsilon^{3/2 - \delta}
$$
as $\varepsilon \rightarrow 0$. 
\end{corollary}
An assumption on the zero distribution in Corollary 1 is inevitable, since
$m'(\varepsilon) \rightarrow 0$ implies that almost all the zeros of
$\zeta(s)$ are simple. Corollary 1 allows one to quantify the 
loss in Levinson's method 
coming from the zeros of $\zeta'(s)$ lying closely to the half-line.
Unfortunately Corollary 1 is a conditional result, and as such it
cannot be used to put a greater proportion of the zeros of $\zeta(s)$ on
the half-line (see \cite{Feng} for related work).

A final consequence of our work
is a criterion for the non-existence of 
the Siegel zero in terms of the zeros of $\zeta'(s)$. We state it only
for completeness since a stronger result has been obtained by Farmer and Ki
\cite{FarmerKi}.
\begin{corollary}
Let $A > 0$. 
If $m'(\varepsilon) \gg \varepsilon^{A}$, for all $\varepsilon > 0$,
then for primitive characters $\chi$ modulo $q$, $$L(1;\chi) > (\log q)^{-18}.$$
for all $q$ sufficiently large. 
\end{corollary}
\begin{proof}
If $m'(\varepsilon) \gg \varepsilon^{A}$ for every $\varepsilon > 0$
then $m(1/4) > 0$ by our Main Theorem, hence $L(1;\chi) > (\log q)^{-18}$ for all $q$ sufficiently
large by
Theorem 1.1 of Conrey-Iwaniec, \cite{ConreyIwaniec}. 
\end{proof}
With some care it is possible to turn the above Corollary into an 
effective result. 
By Dirichlet's formula Corollary 2 also implies that the class number of 
$\mathbb{Q}(\sqrt{-d})$ is at least as large as $ c \sqrt{d} (\log d)^{-18}$ 
with $c$ 
constant. 
%

Farmer and Ki show
that if $m'(\varepsilon) \gg \exp( - \varepsilon^{-1/2 + \delta})$
as $\varepsilon \rightarrow 0$, for some $\delta > 0$, 
then there are $N(T) / \log\log T$
ordinates of zeros of $\zeta(s)$ lying in $[T;2T]$ and such that $(\gamma^{+} - \gamma)\log \gamma = o(1)$ as $T \rightarrow \infty$. Using the result of
Conrey and Iwaniec \cite{ConreyIwaniec} this is enough to rule out the
existence of Siegel zeros. It is an interesting question to determine 
%
whether, given the current technology, one can increase the exponent
$\tfrac 12$ in Farmer and Ki's assumption $m'(\varepsilon) \gg
\exp(-\varepsilon^{-1/2 + \delta})$ and still guarantee the non-existence
of Siegel zeros. 
%
%
%
%
%


\section{Main ideas}

The first part of our Main Theorem follows from the stronger Theorem 1 below.
\begin{theorem} \label{thm2}
Let $A, \delta > 0$. There is a constant $C = C(\delta,A)$ such that
if $0 < \varepsilon < C$ and $m'(\varepsilon) \geq c \varepsilon^{A}$
then $m(\varepsilon^{1/2 - \delta}) \geq (c/8) \varepsilon^{A}$. 
\end{theorem}

The approximate value of $C(\delta,A)$ is $(B \delta / A)^{32A/\delta}$ with $B$ an absolute constant. 
Theorem 1 follows from two technical Propositions which we now describe.
Given a zero $\rho' = \beta' + i\gamma'$ of $\zeta'(s)$ we denote by
$\rho_c = \tfrac 12 + i\gamma_c$ the zero of $\zeta(s)$ lying closest to
$\rho'$. If there are two choices of $\rho_c$ then we pick the one lying
closer to the origin. For any ordinate $\gamma$ of a zero of $\zeta(s)$
we denote by $\gamma^{+}$ the ordinate suceeding $\gamma$
and by $\gamma^{-}$ the ordinate preceeding $\gamma$. We denote by
$\gamma^{\pm}$ the ordinate closest to $\gamma$. 
Theorem 1 follows quickly from the following Proposition. 
\begin{proposition} \label{mainproposition}
Let $0 < \delta, \varepsilon < 1$.
Let $S_{\varepsilon,\delta}(T)$ 
be a set of zeros $\rho' = \beta' + i\gamma'$ of $\zeta'(s)$ 
such that $T \leq \gamma' \leq 2T$, 
$\beta' - \tfrac 12 \leq \varepsilon / \log T$ and
$$|\gamma_c - \gamma_c^{\pm}| > \varepsilon^{1/2 - \delta} / \log T$$
There is a $C = C(\delta, A)$ such that if $0 < \varepsilon < C$
then $|S_{\varepsilon,\delta}(T)| \leq \varepsilon^{A} \cdot T \log T$.
%
\end{proposition}

The proof of Proposition \ref{mainproposition} rests on a Proposition 
describing the structure the roots of $\zeta'(s)$ lying close to the
half-line. The Proposition which we are about to state complements with a
corresponding upper bound the classical lower bound,
$$
|\rho' - \rho_c| \geq \sqrt{\frac{2\big(\beta' - \tfrac 12\big)}{\log T}}
$$
valid for all $\rho' = \beta' + i\gamma'$ (see \cite{Soundararajan}).
%
It might be of independent interest.

\begin{proposition} \label{Proposition1}
Let $0 < \delta < 1$, $0 < \varepsilon < c$ with $c > 0$ an absolute
constant. 
Let $T$ be large and $\mathcal{Z} := \mathcal{Z}_{\varepsilon,\delta}(T)$ be a
set of $\delta/\log T$ well-spaced ordinates of zeros $\rho'$ such that
$\rho' \neq \rho$, 
$\beta' - \tfrac 12 \leq \varepsilon / \log T$ and $T \leq
\gamma' \leq 2T$.
If $|\mathcal{Z}_{\varepsilon,\delta}(T)| \gg \varepsilon^{A} \cdot T \log T$
then, for any given $\kappa > 0$, all but $\kappa |\mathcal{Z}|$ elements 
$\rho' \in \mathcal{Z}$ satisfy the inequality,
$$
\sqrt{\frac{\beta' - \tfrac 12}{\log T}} \ll
|\rho' - \rho_c | 
\ll \sqrt{A \log (\varepsilon \kappa \delta)^{-1}}
\cdot  \sqrt{\frac{\beta' - \tfrac 12}{\log T}}.
$$
\end{proposition}

The proof of the converse part of our Main Theorem builds on ideas of Zhang,
and follows from the following more precise statement valid for any
fixed $\varepsilon > 0$. 
\begin{theorem}
Let $A, \delta > 0$.
There is a $C  = C(\delta, A)$ such that if
$0 < \varepsilon < C$ and $m(\varepsilon^{1/2}) \geq c \varepsilon^{A}$
then $m'(\varepsilon) \geq (c/4) \varepsilon^{A + \delta}$. 
\end{theorem}
The paper is organized as follows. Most of the paper, all the way
until section 7,
is devoted to the proof of the propositions above
and the deduction of Theorem 1 from them.
Following section 7 we prove Theorem 2 and Corollary 1.

\section{Lemma on Dirichlet polynomials}

Define,
$$
A_N(s) := \sum_{n \leq N} \frac{\Lambda(n)W_N(n)}{n^s}
$$
with 
$$ W_N (n) = \begin{cases}
  1 & \text{ for } 1 \leq n \leq N^{1/2} \\
  \log (N / n ) / \log N & \text{ for } N^{1/2} < n \leq N
\end{cases} 
$$
%
The lemma below is due to Selberg.
\begin{lemma}
  \label{selberg}
  Let $\sigma = \tfrac 12 + 2 / \log N$, with $N \leqslant T$.
  Then for $T \leqslant t \leqslant 2 T$,
  \[ \sum_{\rho} \frac{\sigma - \tfrac{1}{2}}{(\sigma - \tfrac{1}{2})^2 + (t -
     \gamma)^2} \ll |A_N(s)| +
     \log T \]
\end{lemma}

\begin{proof}
This is equation $(2.2)$ in \cite{Selberg2}.
\end{proof}
Using the explicit formula we obtain an upper
bound for the number of zeros in a small window $[t-2\pi K/\log t; t + 2\pi K
/\log t]$, in terms of the Dirichlet polynomial
$$
B_N(s) := \sum_{n \leq N} \frac{\Lambda(n)}{n^s} \cdot \left ( 1 - \frac{\log n}{\log N} \right ).
$$
We have the following lemma.
\begin{lemma}\label{explicit}
For $T \leq t \leq 2T$
and $N \leq T$, 
$$
N \big (t + \frac{\pi}{\log N} \big ) - N \big
(t - \frac{\pi}{\log N} \big)
\ll \frac{\log T}{\log N} + \frac{|B_N(\tfrac 12 + it)|}{\log N}
$$
\end{lemma}
\begin{proof}
Let 
$$
F_{\Delta}(v) = \left ( \frac{\sin \pi \Delta v}{\pi \Delta v} \right )^{2}
$$
be the Fejer kernel. 
The Fourier transform of $F_{\Delta}(v)$ is for $|x| < \Delta$
$$
\widehat{F}_{\Delta}(x) := \int_{-\infty}^{\infty} F_{\Delta}(t) e^{- 2\pi i t x}
d x
=
\frac{1}{\Delta} \left ( 1 - \frac{x}{\Delta} \right ) 
$$
and $\widehat{F}_{\Delta}(x) = 0$ for $|x| > \Delta$. 
By the explicit formula (see Lemma 1 in \cite{GoldstonGonek}),
\begin{multline}
\sum_{\gamma} {F}_{\Delta}(\gamma - t) = O(e^{\pi \Delta / 2} \cdot t^{-2} + 1/\Delta) + \frac{1}{2\pi} \int_{-\infty}^{\infty}
{F}_{\Delta}(u - t) \cdot \Re \frac{\Gamma'}{\Gamma} \left ( \frac{1}{4}
+ \frac{i u}{2} \right ) d u \\ - \frac{1}{2\pi}
\sum_{n = 2}^{\infty} \frac{\Lambda(n)}{\sqrt{n}} \left ( \widehat{F}_{\Delta}
\left ( \frac{\log n}{2\pi} \right ) + \widehat{F}_{\Delta} \left (
\frac{-\log n}{2\pi} \right ) \right )
\end{multline}
The integral over $u$ is bounded by $\ll (\log t) / \Delta$. On the other hand the prime sum is bounded by,
$$
\bigg | 
\frac{1}{2\pi \Delta} \sum_{n \leq e^{2\pi \Delta}} \frac{\Lambda(n)}{n^{1/2 + it}}
\cdot \left ( 1 - \left | \frac{\log n}{2\pi \Delta} \right | \right )
\bigg |
$$
Finally $(\pi/2)^2 \sum_{\gamma} {F}_{\Delta}(\gamma - t)$ is an upper
bound for the number of zeros in the interval going from $t - 1 / (2\Delta)$
to $t +  1 / (2\Delta)$. If $T \leq t \leq 2T$ we
choose $2\pi \Delta = \log N$ and we are done.
%
%
\end{proof}
In order to understand the average behavior of the Dirichlet 
polynomials $A_N(s)$ and $B_N(s)$ we use a version of the large sieve.
\begin{lemma} \label{main}
Let $A(s)$ be a Dirichlet polynomial with positive coefficients
and of length $x$. 
Let $s_r = \sigma_r + i t_r$ be points with $T \leq t_r \leq 2T$
and $0 \leq \sigma_r - \tfrac 12 \leq \varepsilon / \log T$ for some
small $\varepsilon > 0$. Suppose that $|t_i - t_j| \geq \delta / \log T$
for $i \neq j$, with $100 \varepsilon < \delta < 1$.
Then, for $x^{k} \leq T$,
$$
\sum_{s_r} | A(s_r) |^{2k}
\leq \frac{20 \log T}{\delta}
\int_{-2T}^{2T} |A(\tfrac 12 + it)|^{2k} dt
$$
\end{lemma}
\begin{proof}
Let $D(s) = A(s)^k$. 
For any $s$ we have, with $\mathcal{C}$ a circle of radius $\delta/(2\log T)$
around $s$,
$$
|D(s)|^2 \leq \frac{4(\log T)^{2}}{\pi\delta^2}
\iint_{\mathcal{C}} |D(x + iy)|^2 d x d y
$$
Summing over all $s = s_r$, since the circles are disjoint we obtain,
$$
\sum_{s_r} |D(s_r)|^2 \leq \frac{4(\log T)^{2}}{\pi \delta^2}
\int_{\tfrac 12 - \delta/\log T}^{\tfrac 12 + \delta/\log T} \int_{T-1}^{2T+1} |D(\sigma + it)|^2 d t
d \sigma
$$
Since the coefficients of $D(s)$ are positive, and $D$ is of length at most $T$, by a majorant principle (see Chapter 3, Theorem 3 in \cite{Montgomery}),
the inner integral is bounded by
$$
\leq 3 e^{2\delta} \int_{-2T}^{2T} |D(\tfrac 12 + it)|^2 dt
$$
Since in addition $\delta < 1$, the claim follows (we obtain a
constant of $8 e^2 / \pi < 20$). 
\end{proof}

Combining the above lemma with Chebyschev's inequality allows us to
understand the average size of the Dirichlet polynomials $A_N(s)$
and $B_N(s)$.

\begin{lemma}
Let $s_r = \sigma_r + i t_r$ be a set of well-spaced points
as appearing in the statement of Lemma 3. Suppose that $N^k \leq T / \log T$.
The number of points $s_r$
for which we have 
$$|A_N(s_r)| > (k/e)\log N \text{ or } |B_N(s_r)| > (k/e)\log N$$
 is bounded 
above by $\ll (e^{-k} / \delta) T \log T$. 
%
\end{lemma}
\begin{proof}
Let $L_N(s)$ be either $A_N(s)$ or $B_N(s)$. 
Let $$
D_N(s) = \sum_{n \leq N} \frac{\Lambda(n)}{n^s}.
$$
By a majorant principle (see Chapter 3, Theorem 3 in \cite{Montgomery}) we have, 
$$
\int_{-2T}^{2T} |L_N(\tfrac 12 + it)|^{2k} dt \leq 3 \int_{-2T}^{2T}
|D_{N}(\tfrac 12 + it)|^{2k} dt
$$
By Soundararajan's lemma 3 in \cite{Soundararajan2}, for $N^{k} \leq T/\log T$ we have,
$$
\int_{-2T}^{2T} |D_N(\tfrac 12 + it)|^{2k} dt \ll k! T (\log N)^{2k}
$$
Therefore, for $N^k \leq T / \log T$, by the previous lemma,
$$
\sum |L_N(s_r)|^{2k} \ll \frac{ k!}{\delta} \cdot T \log T (\log N)^{2k}
$$
It follows that for $N^k \leq T / \log T$, 
the number of points $s_r$ for which $|L_N(s_r)| > B \log N$
is less than,
$$
\ll \left ( \frac{k}{B} \right )^{k} \cdot (T/\delta) \log T
$$
Choosing $B = k/e$ we conclude that the number of points for
which $|L_N(s_r)| > k/e$ is bounded by $(e^{-k} / \delta) T \log T$ as desired.
\end{proof}


\begin{lemma}
  \label{Selberg} 
  Let $0 < c < 1$. Uniformly in $T \leqslant t
  \leqslant 2 T$ and $N \leqslant T$,
$$
\frac{\zeta'}{\zeta}(s) = \sum_{|s - \rho| < c / \log T}
\frac{1}{s - \rho} + O \bigg ( \frac{\log T}{c} \cdot \mathcal{E}_{T,N}(s) 
\bigg )
$$
where
$$
\mathcal{E}_{T,N}(s) :=  \frac{1}{\log N} \cdot \big ( |A_N(s)|
 + |B_N(\tfrac 12 + it)| \big ) 
+ \frac{\log T}{\log N}.
$$
Furthermore, if $s_r$ is a set of well-spaced points as in Lemma 3, 
and $N^{k} \leq T / \log T$, then
$$
\sum_{s_r} |\mathcal{E}_{T,N}(s_r)|^{2k} \ll (k^{2k} / \delta) T \log T.
$$
\end{lemma}

\begin{proof}
Selberg shows in {\cite{Selberg}} (see equation (14) on page 4) that
$$
\frac{\zeta'}{\zeta}(s) = \sum_{|s - \rho| < (\log T)^{-1}}
\frac{1}{s - \rho} + O \bigg ( 
\frac{\log T}{\log N} \cdot |A_N(s)| + \frac{\log^2 T}{\log N} \bigg )
$$
It suffices to notice that the contribution of the zeros $\rho$ with
$c (\log T)^{-1} < |s - \rho| < (\log T)^{-1}$ is bounded above by
\begin{multline*}
\ll \frac{\log T}{c} \cdot \bigg ( N \big (t + \frac{\pi}{\log N} \big )
- N \big ( t - \frac{\pi}{\log N} \big ) \bigg )
\ll \frac{\log T}{c} \cdot \bigg ( \frac{\log T}{\log N} + 
\frac{1}{\log N} \cdot |B_N(\tfrac 12 + it)| \bigg )
\end{multline*}
Combining the above two equations we obtain the first part of the lemma.
Now it remains to estimate the moments of $\mathcal{E}_{T,N}$. 
We have, 
\begin{align*}
\sum_{s_r} |\mathcal{E}_{T,N}(s_r)|^{2k} \ll \bigg (
\frac{C}{\log N} \bigg )^{2k} \cdot \bigg ( 
\sum_{s_r} |A_N(s_r)|^{2k}  + \sum_{s_r} |B_N(s_r)|^{2k} \bigg )
+ ( (C k)^{2k} / \delta) T \log T 
\end{align*}
with $C > 0$ an absolute constant. Using Lemma 3 and proceeding as in Lemma
4 we find that the $2k$-th moments of the Dirichlet polynomials $A_N$ and
$B_N$ is bounded above by $(k! / \delta) T \log T (\log N)^{2k}$. 
Hence we conclude that the $2k$-th moment of $\mathcal{E}_{N,T}$ is bounded
above by $((C k)^{2k} / \delta) T \log T$
%

\end{proof}

\section{Proof of Proposition \ref{Proposition1}}
 
The proof of Proposition \ref{Proposition1} rests on the following
classical lemma. 

\begin{lemma} 
  \label{hadamard}If $\rho' \neq \rho$ then,
  \[ \frac{1}{2} \cdot \log \gamma' = \sum_{\rho} \frac{\beta' - \tfrac 12}
     {(\beta' - \tfrac 12)^2 + (\gamma' - \gamma)^2} + O (1) . \]
\end{lemma}

\begin{proof}
  See Zhang {\cite{Zhang}}, Lemma 3.
\end{proof}

We will show that on average the zero $\rho = \rho_c$ dominates, the claim
then follows shortly.
In order to simplify the notation we define, as in the previous section,
$$
A_N(s) := \sum_{n \leq N} \frac{\Lambda(n)W_N(n)}{n^s}
$$
with $W_N(n)$ the same smoothing as defined in the previous section.
We also define
$$
B_N(s) := \sum_{n \leq N} \frac{\Lambda(n)}{n^s} \cdot \bigg ( 1 -
\frac{\log n}{\log N} \bigg ).
$$
%
On average both Dirichlet polynomials are of size $\log N$. 

\begin{proof}[Proof of Proposition 2]
Let $N \leq T$ to be fixed later. 
In the formula
\begin{equation} \label{formula1}
\tfrac 12 \cdot \log \gamma' = \sum_{\rho} 
\frac{\beta' - \tfrac 12} {(\beta' - \tfrac 12)^2 + (\gamma' - \gamma)^2}
+ O(1)
\end{equation}
The contribution of the
$\rho$'s for which $|\gamma - \gamma'| < \pi (\log N)^{-1}$
is bounded above by
\begin{multline} \label{formula2}
\ll  \bigg ( N \big 
(\gamma' + \frac{\pi}{\log N} \big ) - N \big ( 
\gamma' - \frac{\pi}{\log N} 
\big ) \bigg ) \cdot \frac{\beta' - \tfrac 12}{|\rho_c - \rho'|^2} \\
\ll \bigg ( \frac{\log T}{\log N} + \frac{1}{\log N}\cdot |B_N(\tfrac
 12 + i\gamma')| \bigg ) \cdot \frac{\beta' - \tfrac 12}{|\rho_c - \rho'|^2}.
\end{multline}
by Lemma 2. On the other hand, to bound the 
contribution of the $\rho$'s for which
$|\gamma - \gamma'| > \pi (\log N)^{-1}$ we notice that if $|\gamma' - \gamma| >  \pi (\log N)^{-1}$ then $$(\beta' - \tfrac 12)^2 + (\gamma - \gamma')^2
\gg (2/\log N)^{2} + (\gamma - \gamma')^2 .$$
Therefore the contribution of the
$\rho$'s with $|\gamma - \gamma'| > \pi (\log N)^{-1}$ to (\ref{formula1}) 
is bounded above by 
\begin{multline} \label{formula3}
\ll (\beta' - \tfrac 12) \log N \cdot \bigg ( \sum_{\rho} \frac{2/\log N}{(2/\log N)^2 +
(\gamma - \gamma')^2} \bigg ) \\
\ll (\beta' - \tfrac 12) \log N \cdot \left ( 
\log T + \left | 
A_N \left ( \frac 12 + \frac{1}{\log N} + i\gamma ' \right ) \right | \right ) 
\end{multline}
by Lemma 1. 
Combining (\ref{formula1}), (\ref{formula2}) and (\ref{formula3}) we
conclude that
\begin{multline}
\log T \ll \bigg ( \frac{\log T}{\log N} + \frac{1}{\log N} \cdot
|B_N(\tfrac 12 + i\gamma')| \bigg ) \cdot 
\frac{\beta' - \tfrac 12}{|\rho_c - \rho'|^2} + \\ + (\beta' - \tfrac 12)
\log N \cdot \left ( \log T + \left | 
A_N \left (\frac 12 + \frac{1}{\log N} + i \gamma' \right ) \right | \right )
\end{multline}
Suppose that $N^k \leq T / \log T$ with a $k$ to be fixed later and $N$
the largest integer such that $N^k < T / \log T$. 
By Lemma 4 the number of $\rho' \in \mathcal{Z}_{\varepsilon,\delta}$ 
for which $|B_N(\tfrac 12 + i\gamma')| > (k/e)\log N$ is bounded above
by $c (e^{-k} / \delta) T \log T$ with $c$ a constant. 
Similarly the number of 
$\rho' \in \mathcal{Z}_{\varepsilon,\delta}$ 
for which $|A_N(\tfrac 12 + 1/\log N + i\gamma')|
 > (k/e) \log N$ is also bounded by above by $c (e^{-k} / \delta) T \log T$. 
Choose $k$ so that $c e^{-k} / \delta T \log T \leq (\kappa / 2) |\mathcal{Z}_{\varepsilon,\delta}|$.
Since $|\mathcal{Z}_{\varepsilon,\delta}| 
\geq c_1 \varepsilon^{A} T \log T$ we can take
$k$ to be the closest integer to $c_2 A \log (\kappa \varepsilon \delta)^{-1}$
with $c_2$ an absolute constant. Choose $N$ to be the largest integer
such that $N^k \leq T / \log T$. 
With this choice of $k$ and $N$ it follows
that for at most $\kappa |\mathcal{Z}_{\varepsilon,\delta}|$ elements $\rho' \in \mathcal{Z}_{\varepsilon,\delta}$ we have 
$|B_N (\tfrac 12 + i\gamma')| \geq (k/e) \log N$
or $|A_N(\tfrac 12 + i\gamma')| \geq (k/e) \log N$.
It follows that for all but at most $\kappa |\mathcal{Z}_{\varepsilon,\delta}|$ of the $\rho' \in \mathcal{Z}_{\varepsilon,\delta}$
we have,
$$
c \log T \leq k \cdot \frac{\beta' - \tfrac 12}{|\rho_c - \rho'|^2}
 + \frac{1}{k} \cdot (\beta' - \tfrac 12) \cdot (\log T)^2
$$
with $c > 0$ an absolute constant. If $\varepsilon$ is choosen so that
$\varepsilon < (c/2)$ then (since $\beta' - \tfrac 12 < \varepsilon / \log T$) 
we obtain
$$
(c/2)\log T \leq k \cdot \frac{\beta' - \tfrac 12}{|\rho_c - \rho'|^2}
$$
hence $|\rho_c - \rho'|^2 \leq (k (\beta' - \tfrac 12) / \log T )^{1/2}$
which gives the desired bound for all but at most $\kappa |\mathcal{Z}_{\varepsilon,\delta}|$ elements
$\rho' \in \mathcal{Z}_{\varepsilon,\delta}$. (Recall that $k \ll A \log (\varepsilon \delta
\kappa)^{-1})$). 
\end{proof}

\section{Proof of Proposition 1}



The lemma below is critical, in that it allows us to produce a
sufficiently dense
well-spaced sequence of zeros of $\zeta'(s)$.

\begin{lemma}[Soundararajan \cite{Soundararajan}]
  \label{sound}Suppose that $\rho_1 = \tfrac 12 + \mathi \gamma_1$ and
  $\rho_2 = \tfrac 12 + \mathi \gamma_2$ are two consecutive zeros of $\zeta (s)$
  with $T \leqslant \gamma_1 < \gamma_2 \leqslant 2 T$ for large $T$. Then the
  box,
  \[ \{s = \sigma + \mathi t : \tfrac 12 \leqslant \sigma < \tfrac 12 + 1 / \log T,
     \gamma_1 < t < \gamma_2 \} \]
  contains at most one zero (counted with multiplicity) of $\zeta' (s)$.
\end{lemma}

\begin{proof}
  The only way that $\rho'$ can lie on the critical line is if $\rho' = \rho$.
  Since $\gamma_1 < t < \gamma_2$ this possibility is excluded. As for the box
  $\tfrac 12 < \sigma < \tfrac 12 + 1 / \log T$ we know by Soundararajan's work
  {\cite{Soundararajan}} (see Proposition 6) that the box $\tfrac 12 < \sigma < \tfrac 12 + 1 / \log T$ 
  with $t$ in $[\gamma_1, \gamma_2]$ can contain at most on
  zero of $\zeta'(s)$, counted with multiplicity.
\end{proof}

We are now ready to prove Proposition \ref{mainproposition}.

\begin{proof}[Proof of Proposition 1]

Suppose that $S = S_{\varepsilon , \delta}(T) > \varepsilon^{A}\cdot
T \log T$.  We will show that this leads to a contradiction when
$ 0 < \varepsilon < C(\delta, A)$ with $C(\delta, A)$ some explicit
constant depending only on $\delta$ and $A$ (for example we could take
$C(\delta, A) = 
(c \delta / A)^{32A / \delta}$ with $c > 0$ an absolute constant). 
Since each $\rho' \in S$ satisfies 
$\gamma_c^{-} \leq \gamma' \leq
\gamma_c^{+}$ and $|\gamma_c - \gamma_c^{\pm}| > \varepsilon^{1/2 - \delta}
/ \log T$ by the above lemma for each $\rho' \in S$ there is at most
one zero of $\zeta'(s)$ in $[\gamma_c^{-}, \gamma_c]$ and at most
one zero of $\zeta'(s)$ in $[\gamma_c, \gamma_c^{+}]$. 

We construct a subset $S'$ of $S$ by skipping every second element in $S$.
This produces a subset of at least $(1/2)|S|$ elements, with the property
that the ordinates of elements of $S'$ 
are $\varepsilon^{1/2 - \delta} / \log T$ well-spaced, because
$|\gamma_c - \gamma_c^{\pm}| \geq \varepsilon^{1/2 - \delta} / \log T$
for each $\rho' \in S$. 

By Proposition 2, we have for at least half of the $\rho' \in S'$,
\begin{equation}\label{cineq}
|\gamma' - \gamma_c| \leq |\rho' - \rho_c| \leq \frac{C \sqrt{ A \varepsilon
\log (\varepsilon)^{-1}}}{\log T}.
\end{equation}
with $C > 0$ an absolute constant. 
We call $S''$ the subset of $S'$ satisfying the above inequality.
Since $|\gamma_c^{\pm} - \gamma_c| > \varepsilon^{1/2 - \delta}
/ \log T$ for each $\rho' \in S''$ 
the interval $|\gamma' - t| \leq \varepsilon^{1/2 - \delta}
/ \log T$ contains exactly one ordinate of a zero of $\zeta(s)$ (namely 
$\gamma_c$) once $\varepsilon$ is choosen so small so
as to make the right-hand side of (\ref{cineq}) less than
$\varepsilon^{1/2 - \delta} / \log T$
(for example $\varepsilon < (\delta / CA)^{2/\delta}$ 
would suffice). 

 Using Lemma 5, we have at $s = \rho' \in S$, 
$$
\sum_{|s - \rho| < c / \log T} \frac{1}{s - \rho}
\ll \frac{\log T}{c} \cdot |\mathcal{E}_{T,N}(s)|
$$
Choose $s = \rho' \in S''$, $c = \varepsilon^{1/2 - \delta}$
and $N$ the largest integer such that $N^{k} \leq T / \log T$ 
with a $k$ to be fixed later (ultimately $k = \lceil (A + 1)/\delta \rceil$). 
By our previous remark the left-hand side of the above
expression consists of only one term $(\rho' - \rho_c)^{-1}$. 
Raising the above expression to the $2k$-th power and
then summing over all $\rho' \in S''$ we obtain
\begin{align} \label{mainer}
\sum_{\rho' \in S'} \frac{1}{|\rho' - \rho_c|^{2k}}
& \ll \varepsilon^{-k + 2k\delta} \cdot (C \log T)^{2k}
\sum_{\rho' \in S'} |\mathcal{E}_{T,N}(\rho')|^{2k} \\ 
& \ll \varepsilon^{-k + 2k\delta} \cdot ((C k)^{2k} / \varepsilon^{1/2 - \delta}) 
\cdot T (\log T)^{2k + 1} 
\end{align}
by Lemma 5, 
with $C > 0$ an absolute constant (not necessarily the same in each
occurence). Since for each $\rho' \in S''$ we have,
$$
|\rho' - \rho_c | \ll \frac{\sqrt{A \varepsilon \log (\varepsilon)^{-1}}}{\log T}
$$
the left-hand side of (\ref{mainer}) is at least
\begin{align} \label{mainer2}
\sum_{\rho' \in S'} \frac{1}{|\rho' - \rho_c|^{2k}} & \gg
|S''| \cdot (C/A)^k \cdot 
\varepsilon^{-k} (\log (\varepsilon)^{-1})^{-k} (\log T)^{2k} \\
& \gg (C/A)^k \cdot \varepsilon^{A - k} \cdot (\log (\varepsilon)^{-1})^{-k} \cdot 
T (\log T)^{2k + 1}
\end{align}
since $|S''| \gg \varepsilon^{A} T \log T$. 
Combining the upper bound (\ref{mainer}) and the lower bound (\ref{mainer2}) 
we get
$$
\varepsilon^{A - k} (\log (\varepsilon)^{-1})^{-k} 
\leq \varepsilon^{-k - 1/2 + (2k + 1)\delta } \cdot (C A k)^{2k}
$$
with $C > 0$ an absolute constant.
The above inequality simplifies to
$$
\varepsilon^{A + 1/2}  
\leq (C A k)^{2k} \cdot \varepsilon^{(2k + 1)\delta} \cdot (\log (\varepsilon)^{-1}
)^{k}.
$$
Using the inequality $(\log x) \leq x^{\delta} / \delta$ we obtain
$$
\varepsilon^{A + 1/2} \leq (C A k / \delta)^{2k} \cdot \varepsilon^{k \delta}
$$
Choosing $k$ to be the smallest integers with $k \delta > A + 1$ we obtain
a contradiction once $\varepsilon < (2C \delta^2 / A^2)^{16A / \delta}$
with $C$ an absolute constant.  
(Note: We have certainly not tried to optimize the constant $C(\delta, A)$). 
\end{proof}

\section{Proof of Theorem \ref{thm2}.}
Let $T$ be large. By assumption each interval $[T;2T]$
contains at least $c \varepsilon^{A} N(T)$ ordinates $T \leq \gamma'
\leq 2T$ with 
$\beta' - \tfrac 12 < \varepsilon / \log T$.
If $\rho' = \rho$
for more than half of these zeros of $\zeta'(s)$, 
then we have $\geq (c/2)\varepsilon^{A} N(T)$
zeros $\rho$ with $\gamma^{+} = \gamma$ and so we are done.

Thus we can assume that there are $\geq (c/2) \varepsilon^{A}
N(T)$ zeros $\rho'$ with $T \leq \gamma' \leq 2T$, $\rho' \neq \rho$
and $\beta' - \tfrac 12 < \varepsilon / \log T$. We call the
set of such $\rho'$ by $S$. By Lemma 7
between any two consecutive zeros of $\zeta(s)$ there is at most one
$\rho' \in S$.
For each $\rho' \in S$ consider
two possibilities
\begin{enumerate}
\item $|\gamma_c^{\pm} - \gamma_c| \leq \varepsilon^{1/2 - \delta} / \log T$
\item $|\gamma_c^{\pm} - \gamma_c| > \varepsilon^{1/2 - \delta} / \log T$
\end{enumerate}
Call $S_2$ the subset of $S$ for which the second possibility holds. 
If the second possibility holds for at least one half of the elements in $S$
then $|S_2| \geq (c/2) \varepsilon^{A} T \log T$. But this is impossible
by Proposition 1 once $\varepsilon$ is less than $(c/4) C(\delta,A + 1)$,
with $C(\delta,A)$ as in the statement of Proposition 1.  
Therefore the second possibility can hold for \textit{at most}
one half of the elements in $S$. Hence the first possibility holds
for \textit{at least} a half of the elements in $S$. Call $S_1$
the subset of $S$ for which the first possibility holds. 

By Lemma 7, there are no two $\rho' \in S_1$ lying between the same tuple of
consecutive zeros of $\zeta(s)$. 
Every $\rho' \in S_1$ lies either between $[\gamma_c^{-},\gamma_c]$
or $[\gamma_c;\gamma_c^+]$ and moreover one of these
intervals is of length $\leq \varepsilon^{1/2 - \delta} / \log \gamma_c$. 
Skipping every second $\rho' \in S_1$ we make sure that no
two $\rho_1 \in S_1$ and $\rho_2 \in S_1$ lie between
the same set of consecutive zeros. Therefore every second
$\rho' \in S_1$ gives rise to one (new) zero $\gamma$ (namely
$\gamma_c$ or $\gamma_c^-$)
with $(\gamma^+ - \gamma) \log \gamma \leq \varepsilon^{1/2 - \delta}$. 
Thus we have at least $(1/2) |S_1| \geq (c/8) \varepsilon^{A} \cdot T \log T$
zeros $T \leq \gamma \leq 2T$ such that $(\gamma^+ - \gamma) \log \gamma
\leq \varepsilon^{1/2 - \delta}$.

\section{Lemma: Zeros of the Riemann zeta-function}

In this section we collect a few facts concerning the zeros of 
the Riemann zeta-function. They will be used in the proof of
Theorem 2 and Corollary 1. We first need Gonek's lemma.
\begin{lemma}[Gonek \cite{Gonek}]
If $x = a/b \neq 1$ and
$a,b \leq N$, then,
$$
\sum_{T \leq \gamma \leq 2T} x^{i\gamma} \ll N \log^2 T
$$
\end{lemma}
\begin{proof}
As noted by Ford and Zaharescu (Lemma 1, \cite{FordZaharescu}), 
it follows from Gonek's work that,
$$
\sum_{T \leq \gamma \leq 2T} x^{1/2 + i\gamma} =
-\frac{\Lambda(n_x)}{2\pi} \frac{e^{i T \log (x/n_x)} -1}{i \log (x/n_x)}
+ O \bigg ( x \log^2 (2 x T) + \frac{\log 2T}{\log x} \bigg ).
$$
Since $x$ is not an integer we have $x \neq n_x$. Therefore
the closest that $|x/n_x| = |a / (b n_x)|$ can be to $1$ is when
$b n_x $ is equal to $a \pm 1$. This shows that
$|\log (x/n_x)| \gg a^{-1} \gg N^{-1}$. Therefore the main
term in the above equation is bounded by $N \log T$, 
This gives a bound of $
\sum_{T \leq \gamma \leq 2T} x^{i \gamma}
\ll N / \sqrt{x} \log T
 + \sqrt{x} \log^2 T
$ for $x > 1$. For $x < 1$ this bound is reversed
to $\sqrt{x} N \log T + \log^2{T} / \sqrt{x}$. In either case the final
bound is $\ll N \log^2 T$ because $N^{-1} \leq |x| \leq N$. 
\end{proof}
An quick consequence of the above lemma is a bound for Dirichlet
polynomials.
\begin{lemma}
Let $B_N(s)$
be as in Lemma 2. If $N^k \leq \sqrt{T}$ then,
$$
\sum_{T \leq \gamma \leq 2T} |B_N(\tfrac 12 + i\gamma)|^{2k} \ll
(C k)^{k} \cdot T \log T \cdot (\log N)^{2k} 
$$
for some absolute constant $C > 0$. 
\end{lemma}
\begin{proof}
First notice that for $T \leq t \leq 2T$
\begin{align*}
\sum_{\substack{p^k \leq N \\ k > 1}} \frac{\log p}{p^{k/2 + k i t}} \cdot
\bigg ( 1 - \frac{\log p^k}{\log N} \bigg )
& = \frac{1}{2\pi i}
\int_{2 - i\infty}^{2 + i\infty}
\frac{\zeta'}{\zeta}(s + 1 + 2 i t)
\cdot \frac{N^{s/2} ds}{s^2 \log \sqrt{N}} + O(1) \\
& = - \frac{N^{- it}}{2 t^2 \log N} + \frac{\zeta'}{\zeta} ( 1 + 2 it)
+ O \bigg ( 1 + \frac{\log T}{\log N} \cdot N^{-1/8} \bigg )
\end{align*}
and that the above expression is less than $\ll \log\log T$ by a classical
estimate for the size of $\zeta'/\zeta$ on the Riemann Hypothesis. Therefore,
$$
\sum_{T \leq \gamma \leq 2T} |B_N(\tfrac 12 + i\gamma)|^{2k}
\ll C^{k} \sum_{T \leq \gamma \leq 2T} \bigg | \sum_{p \leq N} \frac{\log p}{p^{1/2 + i\gamma}} \cdot \bigg ( 1 - \frac{\log p}{\log N} \bigg ) \bigg |^{2k} + 
T \log T \cdot (C \log\log T)^{2k}.
$$
with $C > 0$ some absolute constant. We denote the coefficients of the
Dirichlet polynomial over primes by $a(p)$. We have,
\begin{align*}
\sum_{T \leq \gamma \leq 2T} \bigg | \sum_{p \leq N} a(p)p^{-i\gamma} \bigg |^{2k}
& = \sum_{\substack{p_1,\ldots,p_k \leq N \\ q_1, \ldots, q_k \leq N}}
a(p_1)\ldots a(p_k) a(q_1) \ldots a(q_k)
\sum_{T \leq \gamma \leq 2T} \bigg ( \frac{p_1 \ldots p_k}{q_1 \ldots q_k}
\bigg )^{i\gamma}
\end{align*}
The diagonal terms $p_1 \ldots p_k = q_1 \ldots q_k$ contribute at most
$$
\ll T \log T \cdot k! \cdot \bigg ( 2 \sum_{p \leq N} |a(p)|^2 
\bigg )^{k} \ll (C k)^{k} T \log T \cdot (\log N)^{2k}
$$
because given $q_1,\ldots,q_k$ all the solutions to the
equation $p_1 \ldots p_k
= q_1 \ldots q_k$ are obtained by pairing together
each prime $p_i$ with some other prime $q_j$, and there is at most
$k!$ such pairings. To bound the off-diagonal terms $p_1 \ldots p_k
\neq q_1 \ldots q_k$ we notice that $p_1 \ldots p_k \leq N^{k} \leq \sqrt{T}$
and similarly that $q_1 \ldots q_k \leq N^k \leq \sqrt{T}$. Therefore
by Gonek's lemma
$$
\sum_{T \leq \gamma \leq 2T} \bigg ( \frac{p_1 \ldots p_k}{q_1 \ldots q_k}
\bigg )^{i\gamma} \leq \sqrt{T} \log^2 T.
$$
Since
$\sum_{p \leq N} a(p) \ll \sqrt{N}$ it follows that 
the off-diagonal terms contribute at most $C^{k} N^k \cdot \sqrt{T} \log^2 T
\ll C^k T \log^2 T$, 
which is less than the main term as soon as $k > 0$
\end{proof}
An immediate consequence of the above lemma is the following.
\begin{lemma}
Let $T \leq t \leq 2T$. Then,
$$
\sum_{T \leq \gamma \leq 2T} \bigg | N \big ( \gamma + \frac{2\pi}{\log T}
\big ) - N \big ( \gamma - \frac{2\pi}{\log T} \big ) \bigg |^{2k}
\ll (C k)^{2k} \cdot T \log T.
$$
with $C > 0$ an absolute constant. 
\end{lemma}
\begin{proof}
Let $N$ be the largest integer such that $N^k \leq \sqrt{T}$. We have
\begin{align*}
N \big ( \gamma + \frac{2\pi}{\log T} \big ) - N \big ( \gamma - \frac{2\pi}
{\log T} \big ) & \leq N \big ( \gamma + \frac{\pi}{\log N} \big )
- N \big ( \gamma - \frac{\pi}{\log N} \big ) \\
& \ll \frac{\log T}{\log N} + \frac{|B_N(\tfrac 12 + i\gamma)|}{\log N}
\end{align*}
by Lemma 2. Raising the above expression to the $2k$-th power
and then summing over all $T \leq \gamma \leq 2T$ we obtain
$$
\sum_{T \leq \gamma \leq 2T} \bigg | N \big ( \gamma + \frac{2\pi}{\log T}
\big ) - N \big ( \gamma - \frac{2\pi}{\log T} \big ) \bigg |^{2k}
\ll (C k)^{2k} \cdot T \log T + \frac{C^{2k}}{(\log N)^{2k}} 
\sum_{T \leq \gamma \leq 2T} 
|B_N(\tfrac 12 + i\gamma)|^{2k}
$$
with $C > 0$ an absolute constant.
By the previous lemma the sum over $T \leq \gamma \leq 2T$ is
$\ll (C k)^{k} \cdot T \log T \cdot (\log N)^{2k}$ and so the claim follows.
\end{proof}

\begin{corollary}
Let $A > 0$ and $\delta > 0$ be given. 
If $0 < \varepsilon < C(\delta,A)$, with $C(\delta, A)$ depending
only on $\delta$ and $A$, then,
$$
\# \bigg \{ T \leq \gamma \leq 2T : 
N\big ( \gamma + \frac{2\pi}{\log T} \big ) - N \big ( \gamma - 
\frac{2\pi}{\log T} \big ) > \varepsilon^{-\delta}
\bigg \} \leq \varepsilon^{A + 1} \cdot T \log T.
$$
\end{corollary}
\begin{proof}
By the previous lemma we have for $k > 1$, 
$$
\sum_{T \leq \gamma \leq 2T} \bigg | N \big ( \gamma + \frac{2\pi}{\log T} 
\big ) - N \big ( \gamma - \frac{2\pi}{\log T} \big ) \bigg |^{2k}
\ll (C k)^{2k} \cdot T \log T
$$
with $C > 0 $ a positive absolute constant. 
Therefore the number of $T \leq \gamma \leq 2T$ for which
the interval $[\gamma - 2\pi / \log T; \gamma + 2\pi / \log T]$
contains more than $\varepsilon^{-\delta}$ zeros
is bounded above by $\varepsilon^{2k \delta} (C k)^{2k} \cdot T \log T$. Choose
$k = \lceil A / \delta \rceil$. Then $\varepsilon^{2k\delta} (C k)^{2k} 
\leq \varepsilon^{A}$ 
provided that $\varepsilon \leq (c A / \delta)^{-4 / \delta})$
with $c > 0$ an absolute constant. 
\end{proof}

\section{Proof of Theorem 2}

We will require the following two lemma.
\begin{lemma}[Zhang \cite{Zhang}]
Let $\varepsilon < 1$. If $\rho = \tfrac 12 + i\gamma$ is a 
zero of $\zeta(s)$ such that $\gamma$ is sufficiently large
and $(\gamma^{+} - \gamma) \log \gamma < \varepsilon$
then there exists a zero $\rho'$ of $\zeta'(s)$ such that
$$
|\rho' - \rho| \leq \frac{2 \varepsilon}{\log \gamma}.
$$
\end{lemma}
\begin{lemma}[Soundararajan \cite{Soundararajan}]
We have,
$$
|\rho' - \rho_c|^2 \geq \frac{2 \big ( \beta' - \tfrac 12 \big )}
{\log \gamma'}.
$$
\end{lemma}
We are now ready to prove Theorem 2. 
\begin{proof}[Proof of Theorem 2]
Suppose that there are at least $c \varepsilon^{A} \cdot
T \log T$ 
zeros $T \leq \gamma \leq 2T$ such that $(\gamma^{+} - \gamma)\log \gamma
\leq \varepsilon^{1/2}$. Call this set $S$. 
If $\gamma^{+} = \gamma$ for at least
a half of the elements in $S$ then $\rho' = \rho$ and hence $\beta'
= \tfrac 12$ for at least
$(c/2) \varepsilon^{A} \cdot T \log T$ zeros.

Hence suppose that $\gamma^+ > \gamma$ for at least half of the elements in
$S$ and call the subset of such elements $S_1$. 
By Corollary 3,
the number of $T \leq \gamma \leq 2T$ such that
the interval $[\gamma - 2\pi / \log T; \gamma + 2\pi / \log T]$
contains more than $\varepsilon^{-\delta}$ zeros is
$\leq (c/4) \varepsilon^{A} \cdot T \log T$, provided that
$\varepsilon$ is small enough with respect to $\delta$ and $A$. 
Therefore there is a subset $S_2$ of $S_1$ of cardinality
$\geq (c/4) \varepsilon^{A} \cdot T \log T$ with the properties that
$0 < (\gamma^+ - \gamma) \log \gamma < \varepsilon^{1/2}$
and the number of zeros in the interval
$[\gamma - 2\pi / \log T, \gamma + 2\pi / \log T]$
is less than $\varepsilon^{-\delta}$. 

By Lemma 10 each $\rho \in S_2$ gives rise to a zero $\rho'$
such that $|\rho' - \rho| \leq 2\sqrt{\varepsilon} / \log T$. By Lemma 11
the zero $\rho'$ satisfies $(\beta' - \tfrac 12)\log \gamma \leq \varepsilon$.
Furthermore the interval $|t - \gamma| < 2 \sqrt{\varepsilon} / \log T$ 
contains
at most $\varepsilon^{-\delta}$ zero. Therefore striking out at most
$\varepsilon^{-\delta}$ zeros from $S_2$ we obtain each time
a new and distinct zero $\rho'$ of $\zeta'(s)$. It follows that 
$\varepsilon^{\delta} |S_2|$ is a lower bound for the number of zeros
$\rho'$ with $(\beta' - \tfrac 12) \log \gamma \leq \varepsilon$.
Hence $m'(\varepsilon) \geq (c/4) \varepsilon^{A + \delta}$, as desired.
\end{proof}

\section{Proof of Corollary 1}
 
The Pair Correlation Conjecture asserts that the number of
zeros $T \leq \gamma_1, \gamma_2 \leq 2T$ for which
$2\pi \alpha / \log T < \gamma_1 - \gamma_2 \leq 2 \pi \beta / \log T$ 
is asymptotically
$$
N(T) \cdot \int_{\alpha}^{\beta} \bigg ( 1 - \bigg ( \frac{\sin(\pi u)}{\pi u}
\bigg )^{2} + \delta(u) \bigg ) d u
$$
with $\delta$ denoting Dirac's delta function. 
Here we derive a simple consequence of the Pair Correlation Conjecture
for small gaps between \textit{consecutive} zeros of the Riemann zeta-function. 
The lower bound
is not
optimal but sufficient for our needs.
\begin{lemma}
Assume the Pair Correlation Conjecture.
Let $\delta > 0$ be given. 
Then $\varepsilon^{3 + \delta} \ll m(\varepsilon) \ll \varepsilon^3$
provided that $0 < \varepsilon < C(\delta)$ with $C(\delta)$ a constant
depending only on $\delta$. 
\end{lemma}
\begin{proof}
The Pair Correlation Conjecture asserts that the number of distinct zeros
$T \leq \gamma_1, \gamma_2 \leq 2T$ for which
$ 0 \leq \gamma_1 - \gamma_2 \leq 2\pi \alpha / \log T$ is
asymptotically $N(T) \cdot f(\alpha)$ with $f(\alpha)$
such that $f(\alpha) \sim c \cdot \alpha^3$ as $\alpha \rightarrow 0$. 
The number of $T \leq \gamma \leq 2T$ such that 
$(\gamma^+ - \gamma) \log \gamma \leq \varepsilon$ is less than the number of
distinct $T \leq \gamma_1, \gamma_2 \leq 2T$ for which $0 \leq
\gamma _1 - \gamma_2 \leq 2\pi \varepsilon / \log T$ therefore
$m(\varepsilon) \leq f(\varepsilon) \ll \varepsilon^3$.

Now consider the set of $T \leq \gamma_1, \gamma_2 \leq 2T$
for which $\tfrac {\varepsilon}{2} \leq (\gamma_1 - \gamma_2)\log \gamma_1
\leq \varepsilon$. Call $S$ the set of $T \leq \gamma_1 \leq 2T$
for which the interval $[\gamma_1 - 2\pi / \log T; \gamma_1 + 2\pi / \log T]$
contains at most $\varepsilon^{-\delta}$ zeros. By Corollary 3, 
the zero $T \leq \gamma_1 \leq 2T$ with $\gamma_1 \notin S$
have cardinality $\leq \varepsilon^{A} \cdot T \log T$ provided
that $0 < \varepsilon < C(\delta, A)$ (we choose $A = 100$ for example). 
We have
\begin{equation} \label{PCC}
\sum_{\substack{T \leq \gamma_1,\gamma_2 \leq 2T \\
\tfrac {\varepsilon}{2} \leq (\gamma_1 - \gamma_2) \log \gamma_1 \leq \varepsilon} }
1
= 
\sum_{\gamma_1 \in S} \sum_{\substack{T \leq \gamma_2 \leq 2T \\
\tfrac {\varepsilon}{2} \leq (\gamma_1 - \gamma_2) \log \gamma_1 \leq \varepsilon}}
1 
+ \sum_{\gamma_1 \not{\in} S} \sum_{\substack{T \leq \gamma_2 \leq 2T \\
\tfrac {\varepsilon}{2} \leq (\gamma_1 - \gamma_2) \log \gamma_1 \leq \varepsilon}}
1
\end{equation} 
Since $\gamma_1 \in S$ there can be at most $\varepsilon^{-\delta}$
zeros $\gamma_2$ satisfying $\varepsilon / 2 \leq |\gamma_1 - \gamma_2| 
\log \gamma_1 \leq \varepsilon$. Therefore the first sum is 
bounded by 
$$
\sum_{\substack{\gamma_1 \in S \\ (\gamma_1^+ - \gamma_1) \log \gamma_1 \leq \varepsilon}}
\varepsilon^{-\delta} \ll \varepsilon^{-\delta} \cdot m(\varepsilon) \cdot T \log T 
$$
because for each $\gamma_1 \in S$ the inner sum over $\gamma_2$
is $\leq \varepsilon^{-\delta}$ if $(\gamma_1^+ - \gamma_1) \log \gamma_1
\leq \varepsilon$ and is $0$ otherwise.
On the other hand the second sum is by Cauchy-Schwarz less than,
\begin{multline*}
|S|^{1/2} \cdot \bigg ( \sum_{T \leq \gamma_1 \leq 2T} \bigg ( \sum_{\substack{
T \leq \gamma_2 \leq 2T \\ \tfrac { \varepsilon}{2} \leq (\gamma_1 - \gamma_2)\log \gamma_1
\leq \varepsilon}} 1 \bigg )^{2} \bigg )^{1/2} \leq \\ \leq
|S|^{1/2} \cdot \bigg ( \sum_{T \leq \gamma_1 \leq 2T} \bigg ( 
N \big ( \gamma_1 + \frac{2\pi}{\log T} \big ) - N \big (
\gamma_1 - \frac{2\pi}{\log T} \big ) \bigg )^{2} \bigg )^{1/2} 
\ll \varepsilon^{A/2} \cdot T \log T 
\end{multline*}
by Lemma 9. By the Pair Correlation Conjecture the left-hand side of 
(\ref{PCC}) is asymptotically $C \cdot N(T) \cdot \varepsilon^3$ for some
absolute constant $C > 0$. Combining the above three equations we get
$C \varepsilon^3 \leq m(\varepsilon) \varepsilon^{-\delta} + C_1 \varepsilon^{A/2}$
for some absolute constant $C, C_1 > 0$. Therefore if $\varepsilon$ is
small enough then $\varepsilon^{3 + \delta} \ll m(\varepsilon)$.
\end{proof}

We are now ready to prove Corollary 1. 

\begin{proof}[Proof of Corollary 1]
By the previous lemma, on the Pair Correlation, 
we have $m(\varepsilon^{1/2}) \gg \varepsilon^{3/2 + \delta}$
for all $C(\delta) > \varepsilon > 0$. 
Therefore by the second part of our Main Theorem we get
$m'(\varepsilon) \gg \varepsilon^{3/2 + \delta}$
for all $C(\delta) > \varepsilon > 0$. 
Now suppose to the contrary that there is a $\eta > 0$ and a 
sequence of $\varepsilon \rightarrow 0$ such that
$m'(\varepsilon) \gg \varepsilon^{3/2 - \eta}$. 
Then, by Theorem 1 on the same sub-sequence of $\varepsilon \rightarrow 0$
we have $m(\varepsilon^{1/2 - \delta}) \gg \varepsilon^{3/2 - \eta}$. However
by the Pair Correlation Conjecture we have $\varepsilon^{3/2 - 3\delta}
\gg m(\varepsilon^{1/2 - \delta}) \gg \varepsilon^{3/2 - \eta}$. 
Choosing $0 < \delta < (1/3)\eta$ and letting $\varepsilon
\rightarrow 0$ along the subsequence, we obtain a contradiction.
\end{proof}

\section{Acknowledgments}

I would like to thank Prof. Farmer, Prof. Ki and 
Prof. Soundararajan 
for comments on an early draft of this paper.

\bibliographystyle{plain}
\bibliography{farmer4}

\begin{thebibliography}{10}

\bibitem{Berndt}
Bruce~C. Berndt.
\newblock The number of zeros for {$\zeta \sp{(k)}\,(s)$}.
\newblock {\em J. London Math. Soc. (2)}, 2:577--580, 1970.

\bibitem{ConreyIwaniec}
B.~Conrey and H.~Iwaniec.
\newblock Spacing of zeros of {H}ecke {$L$}-functions and the class number
  problem.
\newblock {\em Acta Arith.}, 103(3):259--312, 2002.

\bibitem{Mezzadri}
Eduardo Due{\~n}ez, David~W. Farmer, Sara Froehlich, C.~P. Hughes, Francesco
  Mezzadri, and Toan Phan.
\newblock Roots of the derivative of the {R}iemann-zeta function and of
  characteristic polynomials.
\newblock {\em Nonlinearity}, 23(10):2599--2621, 2010.

\bibitem{FarmerKi}
David~W. Farmer and Haseo Ki.
\newblock Landau--{S}iegel zeros and zeros of the derivative of the {R}iemann
  zeta function.
\newblock {\em Adv. Math.}, 230(4-6):2048--2064, 2012.

\bibitem{Feng}
Shaoji Feng.
\newblock A note on the zeros of the derivative of the {R}iemann zeta function
  near the critical line.
\newblock {\em Acta Arith.}, 120(1):59--68, 2005.

\bibitem{FordZaharescu}
Kevin Ford and Alexandru Zaharescu.
\newblock On the distribution of imaginary parts of zeros of the {R}iemann zeta
  function.
\newblock {\em J. Reine Angew. Math.}, 579:145--158, 2005.

\bibitem{Fujii}
Akio Fujii.
\newblock On the zeros of {D}irichlet {$L$}-functions. {II}.
\newblock {\em Trans. Amer. Math. Soc.}, 267(1):33--40, 1981.
\newblock With corrections to: ``On the zeros of Dirichlet $L$-functions. I''
  [Trans. Amer. Math. Soc. {{\bf{1}}96} (1974), 225--235; MR {{\bf{5}}0}
  \#2096] and subsequent papers.

\bibitem{Garaev}
M.~Z. Garaev and C.~Y. Y{\i}ld{\i}r{\i}m.
\newblock On small distances between ordinates of zeros of {$\zeta(s)$} and
  {$\zeta'(s)$}.
\newblock {\em Int. Math. Res. Not. IMRN}, (21):Art. ID rnm091, 14, 2007.

\bibitem{GoldstonGonek}
D.~A. Goldston and S.~M. Gonek.
\newblock A note on {$S(t)$} and the zeros of the {R}iemann zeta-function.
\newblock {\em Bull. Lond. Math. Soc.}, 39(3):482--486, 2007.

\bibitem{Gonek}
S.~M. Gonek.
\newblock An explicit formula of {L}andau and its applications to the theory of
  the zeta-function.
\newblock In {\em A tribute to {E}mil {G}rosswald: number theory and related
  analysis}, volume 143 of {\em Contemp. Math.}, pages 395--413. Amer. Math.
  Soc., Providence, RI, 1993.

\bibitem{Ki}
Haseo Ki.
\newblock The zeros of the derivative of the {R}iemann zeta function near the
  critical line.
\newblock {\em Int. Math. Res. Not. IMRN}, (16):Art. ID rnn064, 23, 2008.

\bibitem{Levinson}
Norman Levinson.
\newblock More than one third of zeros of {R}iemann's zeta-function are on
  {$\sigma =1/2$}.
\newblock {\em Advances in Math.}, 13:383--436, 1974.

\bibitem{Montgomery}
Hugh~L. Montgomery.
\newblock {\em Ten lectures on the interface between analytic number theory and
  harmonic analysis}, volume~84 of {\em CBMS Regional Conference Series in
  Mathematics}.
\newblock Published for the Conference Board of the Mathematical Sciences,
  Washington, DC, 1994.

\bibitem{Selberg}
A.~Selberg.
\newblock On the value distribution of the derivative of the {R}iemann
  zeta-function.
\newblock {\em Unpublished manuscript available at
  http://publications.ias.edu/selberg/section/2483}.

\bibitem{Selberg2}
Atle Selberg.
\newblock On the remainder in the formula for {$N(T)$}, the number of zeros of
  {$\zeta(s)$} in the strip {$0<t<T$}.
\newblock {\em Avh. Norske Vid. Akad. Oslo. I.}, 1944(1):27, 1944.

\bibitem{Soundararajan}
K.~Soundararajan.
\newblock The horizontal distribution of zeros of {$\zeta'(s)$}.
\newblock {\em Duke Math. J.}, 91(1):33--59, 1998.

\bibitem{Soundararajan2}
Kannan Soundararajan.
\newblock Moments of the {R}iemann zeta function.
\newblock {\em Ann. of Math. (2)}, 170(2):981--993, 2009.

\bibitem{Speiser}
Andreas Speiser.
\newblock Geometrisches zur {R}iemannschen {Z}etafunktion.
\newblock {\em Math. Ann.}, 110(1):514--521, 1935.

\bibitem{Zhang}
Yitang Zhang.
\newblock On the zeros of {$\zeta'(s)$} near the critical line.
\newblock {\em Duke Math. J.}, 110(3):555--572, 2001.

\end{thebibliography}

\end{document}